\documentclass[12pt]{amsart}

\usepackage{amssymb,latexsym}
\usepackage{enumerate}
\usepackage{hyperref}
\usepackage{fullpage}
\usepackage{soul}

\usepackage{fixme}
\fxsetup{
    status=draft,
    author=,
    layout=margin,
    theme=color
}

\makeatletter \@namedef{subjclassname@2010}{%
  \textup{2010} Mathematics Subject Classification}
\makeatother

\newcounter{thm} \numberwithin{thm}{section}
\newtheorem{Theorem}[thm]{Theorem}

\newtheorem{Corollary}[thm]{Corollary}

\newtheorem{Claim}{Claim}

\usepackage{tikz}
\usetikzlibrary{decorations.pathreplacing}
\tikzset{mybrace/.style={decoration={brace,raise=1.8mm},decorate}}
\tikzset{mybracedown/.style={decoration={brace,mirror,raise=1.8mm},decorate}}
\usetikzlibrary{math}
\usetikzlibrary{patterns}
\usetikzlibrary{matrix,backgrounds}

\author[O. Roche-Newton]{Oliver Roche-Newton} \address{Institute for Algebra, Johannes Kepler Universit\"{a}t\\
Linz, Austria}
\email{o.rochenewton@gmail.com}

\date{}

\begin{document}

\baselineskip=17pt

\title{A better than $3/2$ exponent for iterated sums and products over $\mathbb R$}

\date{}
\maketitle

\begin{abstract} In this paper, we prove that the bound
\[
\max \{ |8A-7A|,|5f(A)-4f(A)| \} \gg |A|^{\frac{3}{2} + \frac{1}{54}-o(1)}
\]
holds for all $A \subset \mathbb R$, and for all convex functions $f$ which satisfy an additional technical condition. This technical condition is satisfied by the logarithmic function, and this fact can be used to deduce a sum-product estimate
\[
\max \{ |16A| , |A^{(16)}| \} \gg |A|^{\frac{3}{2} + c},
\]
for some $c>0$. Previously, no sum-product estimate over $\mathbb R$ with exponent strictly greater than $3/2$ was known for any number of variables.
Moreover, the technical condition on $f$ seems to be satisfied for most interesting cases, and we give some further applications. In particular, we show that
\[
|AA| \leq K|A| \implies \,\forall d \in \mathbb R \setminus \{0 \}, \,\, |\{(a,b) \in A \times A : a-b=d \}| \ll K^C |A|^{\frac{2}{3}-c'},
\]
where $c,C>0$ are absolute constants.
%This is used to give some information in the direction of an inverse theorem for the Szemer\'{e}di-Trotter Theorem, by giving a better incidence bound for $I(P,L)$ when the point set $P=A \times A$ is a Cartesian product and the corresponding product set $AA$ is sufficiently small.

\end{abstract}

% \setcounter{section}{-1}
% \section{Notation}

% All quantities are assumed to implicitly depend on the size of the set of points as a large parameter. In particular, if the set in question is of the form $A \times A$ or $A 
% \times \cdots \times A$ for some $A$, we assume $|A|$, the size of the set $A$, to be large.

% Throughout we will be using the usual $\lesssim, \gtrsim,\sim $ absorb powers of $\log|A|$ and constants. More formally, $X \lesssim Y$ (equivalently, $Y \gtrsim X$) if
% $$
% X \leq Y \log^C |A| 
% $$
% for some absolute constant $C > 0$. $X \sim Y$ means that $X \lesssim Y$ and $X \gtrsim Y$ simultaneously.  Importantly, the hidden powers of $\log|A|$ and constants may depend on $n$, the maximal number of prime divisors of an element of $A$. We will henceforth imply this without mentioning explicitly.

% Similarly, $X \ll Y$ (equivalently, $Y \gg X$) if 
% $$
% X \leq C Y 
% $$
% for some absolute constant $C > 0$. $X \approx Y$ means that $X \ll Y$ and $X \gg Y$.

\section{Introduction}

Given a finite set $A \subset \mathbb R$, let $A+A$ and $AA$ denote the sum set and product set of $A$ respectively. That is,
\[
A+A:= \{a+b : a,b \in A \}, \,\,\,\,\,\,\,\,
AA:= \{ab : a,b \in A \}.
\]
The Erd\H{o}s-Szemer\'{e}di \cite{ES} sum-product conjecture states that, for all $\epsilon  > 0$, the bound
\begin{equation} \label{ESconj}
\max \{ |A+A|, |AA| \} \geq |A|^{2- \epsilon}
\end{equation}
holds for all sufficiently large $A \subset \mathbb N$. %This fascinating conjecture has received considerable attention, but despite significant progress and several partial results, it remains one of the major open problems in additive combinatorics.

One can also consider an analogue of the Erd\H{os}-Szemer\'{e}di conjecture with more variables. The notation $kA$ and $A^{(k)}$ are used for the $k$-fold sum set and $k$-fold product set respectively. That is,
\[
kA:= \{a_1+ \dots + a_k : a_1,\dots,a_k \in A \}, \,\,\,\,\,\,\,\,
A^{(k)}:= \{a_1 \cdots a_k : a_1,\dots,a_k \in A  \}.
\]
It was also conjectured in \cite{ES} that
\[
\max \{ |kA|, |A^{(k)}| \} \geq |A|^{k- \epsilon}
\]
holds for all sufficiently large $A \subset \mathbb N$, a conjecture which appears to be far out of reach at present.

However, an outstanding result of Bourgain and Chang \cite{BC} establishes that one of the iterated sum set or product set must exhibit unbounded growth. They proved that, for all $h \in \mathbb N$, there exists $k \in \mathbb N$ such that
\begin{equation} \label{BC}
\max \{ |kA|, |A^{(k)}| \} \geq |A|^{h}
\end{equation}
holds for all $A \subset \mathbb N$. See also \cite{ZP} for an alternative and more easily digestible proof of the Bourgain-Chang Theorem.

The sum-product problem can also be studied for $A \subset \mathbb R$, and indeed most contemporary work on the conjectured bound \eqref{ESconj} uses geometric arguments that hold over the reals. The current best estimate towards this conjecture is the bound\footnote{Here and throughout this paper, the notation  $X\gg Y$, $Y \ll X,$ $X=\Omega(Y)$, and $Y=O(X)$ are all equivalent and mean that $X\geq cY$ for some absolute constant $c>0$.}
\[
\max\{|A+A|,|AA| \} \gg |A|^{\frac{4}{3}+ \frac{2}{1167}-o(1)},
\]
due to Rudnev and Stevens \cite{RS}. This estimate holds for all $A \subset \mathbb R$, and no quantitative improvements are known under the additional restriction that $A \subset \mathbb N$.

On the other hand, progress for the real $k$-fold analogue of the Erd\H{o}s-Szemer\'{e}di Conjecture lags far behind the best-known results over $\mathbb N$. For instance, the bound \eqref{BC} is not known to hold even for $h=3/2$ for $A \subset \mathbb R$, at which point the dominant geometric methods appear to reach their limitation.

The main result of this paper gives some progress for this problem, pushing the growth exponent for the $k$-fold sum-product problem beyond the threshold $3/2$.

\begin{Theorem} \label{thm:main1}
   There exists an absolute constant $c>0$ such that, for any finite set $A \subset \mathbb R$
   \[
\max \{ |16A|, |A^{(16)}| \} \geq |A|^{\frac{3}{2}+c}.
\]
\end{Theorem}

Theorem \ref{thm:main1} is a special case of a more general result (see the forthcoming Theorem \ref{thm:main1gen}) concerning the relationship between convexity and sum sets. The guiding idea here is that strictly convex or concave functions disturb any additive structure existing in a set. This idea has been explored in several papers, including the pioneering work of Elekes, Nathanson and Ruzsa \cite{ENR}, who broke new ground using incidence theory.

Recently, a new elementary method was introduced by Solymosi (see \cite{RSSS}) which has enabled further progress. This ``squeezing" technique was used in \cite{HRNR} to prove that the bound
\begin{equation} \label{3/2}
|A+A-A||f(A)+f(A)-f(A)| \gg \frac{|A|^3}{\log|A|}
\end{equation}
holds for all $A \subset \mathbb R$. A further improvement, removing the logarithmic factor, was later given by Bradshaw \cite{B}. This technique has also been used in \cite{HRNS} to give improved bounds for the number of distinct dot products determined by a point set in $\mathbb R^2$. The same squeezing technique is the main tool for this paper, and a closer analysis allows for an improvement to the bound \eqref{3/2} by introducing more variables, provided that the convex function $f$ satisfies an additional technical condition.

As another application of our general result, we prove the following result concerning the number of additive representations for sets with small multiplicative doubling.

\begin{Theorem} \label{thm:main2}
    There exist absolute constants $c,C>0$ such that, for any finite set $A$ of positive reals and any $t \in \mathbb R \setminus \{0 \}$,
    \[
    |AA| \leq K|A| \implies
    |\{ (a,b) \in A \times A : a-b=t \}| \ll K^{C} |A|^{\frac{2}{3}- c} .
    \]
\end{Theorem}
A weaker version of Theorem \ref{thm:main2}, without the constant $c$, can be proven by a simple application of the Szemer\'{e}di-Trotter Theorem. Such results are sometimes referred to in the sum-product community as \textit{threshold bounds}. Most of these bounds have now been improved by refining and augmenting these basic arguments. Theorem \ref{thm:main2} represents an improvement for one of the last standing threshold bounds in sum-product theory.

\subsection{Notation} Henceforth, all sets introduced in this paper are assumed to be finite unless stated otherwise. We use $\log$ to denote the logarithm function with base $2$ and $\ln$ for the logarithmic function with base $e$. As mentioned in an earlier footnote, throughout this paper, the notation  $X\gg Y$, $Y \ll X,$ $X=\Omega(Y)$, and $Y=O(X)$ are all equivalent and mean that $X\geq cY$ for some absolute constant $c>0$. $X \gg_a Y$ means that the implied constant is no longer absolute, but depends on $a$. We occasionally use the symbols $\gtrsim$ and $\lesssim$ to additionally absorb logarithmic factors. That is, $X \gtrsim Y $ and $Y \lesssim X$ both denote that $X \gg Y/(\log Y)^c$ for some absolute constant $c>0$.

%\orn{Possibly also mention the geometric interpretation as a better incidence bound for cartesian product sets $A \times A$ when $A$ has small multiplicative doubling.}

\section{Preliminary results}

The main external result that will be used in this paper is taken from \cite{HRNR}, and concerns the growth of iterated sum sets of $f(A)$ when $f$ is a highly convex function and $A$ has small additive growth.

For an interval $I$, we say $f : I \rightarrow \mathbb R$ is a $0$-convex function if it is strictly monotone on $I$, and in general, $f$ is $k$-convex on $I$ if each of the derivatives $f^{(1)}, \dots ,f^{(k+1)}$ and, for all $x \in I$ and $1 \leq j \leq k+1$, $f^{j}(x) \neq 0$ (hence all but the last derivative are strictly monotone).

\begin{Theorem} \label{thm:BOM}

Let $A$ be a finite set of real numbers contained in an interval $I$ and let $f:I \rightarrow \mathbb R$ be a function which is $k$-convex, for some $k \geq 1$. Suppose that $|A| > 10k$ and $|A + A - A| \leq  K|A|$. Then
\[
|2^kf(A) - (2^k-1)f(A)| \geq \frac{ |A|^{k+1}}{(CK)^{2^{k+1}-k-2}(\log |A|)^{2^{k+2}-k-4}}
\]
for some absolute constant $C > 0$.
\end{Theorem}

We will use the following result of Elekes, Nathanson and Ruzsa \cite{ENR}. 

\begin{Theorem} \label{thm:ENR}
Let $I$ be an interval, let $f:I \rightarrow \mathbb R$ be a strictly convex or concave function, and suppose that $A \subset I$. Then, for all $k \in \mathbb N$,
\[
|kA| |kf(A)| \gg |A|^{3-2^{1-k}}.
\]
\end{Theorem}

Setting $f(x)= \ln x$ in Theorem \ref{thm:ENR} establishes that the bound \eqref{BC} holds over the reals for all $h < 3/2$. The proof of Theorem \ref{thm:ENR} uses a variant of the Szemer\'{e}di-Trotter Theorem applied iteratively.

We will use the Pl\"{u}nnecke-Ruzsa Theorem to control the growth of iterated sum sets.

\begin{Theorem} \label{thm:Plun}
Let $X$ and $Y$ be sets in an abelian group. Then for any non-negative integers $k$ and $l$ such that $k+l > 1$,
\[
|kX-lX| \leq \frac{ |X+Y|^{k+l}}{|Y|^{k+l-1}}.
\]
\end{Theorem}

We will also use the following form of the Ruzsa Triangle Inequality.

\begin{Theorem} \label{thm:Ruzsa}
For any sets $X,Y,Z$ in an abelian group,
\[
|Y-Z|\leq\frac{|X+Y||X+Z| }{|X|}.
\]
\end{Theorem}

\section{The main general result}

All of the results mentioned in the introduction are derived from the forthcoming Theorem \ref{thm:maingen}, and the purpose of this section is to state and prove this result.

Before doing so, we need to introduce the function $f_d$, which is central to this paper. Let $a \in \mathbb R$ and let $I$ denote the interval $I=(a, \infty)$. Suppose that $f: I \rightarrow \mathbb R$ is a strictly convex or concave function. Let $d$ be a strictly positive real number. Define the function
\[
f_d: I \rightarrow J, \,\,\,\,\, f_d(x):=f(x+d)-f(x).
\]
Here $J$ denotes the range of $f_d$, and $J$ is always an interval. Since $f$ is strictly convex or concave, it follows that the function $f_d$ is monotone. In particular, the inverse function $f_d^{-1}:J \rightarrow I$ is defined.

%\orn{Introduce $f_d$ notation up here and observe the basic fact that $f_d$ is monotone follows from $f$ being strictly convex/concave. In particular, the function $f_d^{-1}: J \rightarrow I$ exists.}

\begin{Theorem} \label{thm:maingen}
    Let $a \in \mathbb R$ and let $I$ denote the interval $I=(a, \infty)$, and let $m \in \mathbb N$. Suppose that $A \subset I$ is a finite set. Suppose that $f:(a, \infty) \rightarrow \mathbb R$ is a strictly convex or concave function.  Suppose also that, for any $d \in (0, \infty)$, the function $f_d^{-1}: J \rightarrow I$ has the property that its first three derivatives have a combined total of at most $m$ zeroes. Then
\begin{equation} \label{product}
|8A-7A|^{16}|5f(A)-4f(A)|^{11}  \gg_m \frac{|A|^{41}}{ (\log|A|)^{77}}.
\end{equation}
In particular,
\begin{equation} \label{max}
 \max \{|8A-7A|, |5f(A)-4(A)| \} \gtrsim_m |A|^{\frac{3}{2} +\frac{1}{54}}.
\end{equation}
  
\end{Theorem}

\begin{proof}

The inequality \eqref{max} follows from \eqref{product} by the pigeonhole principle. Therefore, it is sufficient to prove \eqref{product}.

We will consider the case where $f$ is strictly convex only. For strictly concave $f$, the same proof works with some minor modifications. Write $A=\{a_1< \dots < a_n \}$ and define
\[
D:=\{ a_{i+1} - a_i : 1 \leq i \leq n-1 \}
\]
to be the set of consecutive differences determined by $A$. We will begin by performing a dyadic decompositions in order to regularise the set $D$.

For $d \in D$, define
\[
A_d:= \{ a_i \in A : a_{i+1} - a_i=d \}.
\]
Observe that
\[
\sum_{d \in D} |A_d| = n-1.
\]
After a dyadic decomposition and pigeonholing, it then follows that there exists $D' \subset D$ and some $K \in \mathbb N$ such that
\[
K \leq |A_d|< 2K, \,\,\,\, \forall \, d \in D'
\]
and
\begin{equation} \label{D'Kbound}
|D'|K \gg \frac{ |A|}{ \log |A|}.
\end{equation}

For simplicity of notation, we assume that $|A_d|=K$, and discard any additional elements if necessary (strictly speaking, we are passing to another subset here, but we do not introduce any new notation for this subset). We also make the simplifying assumption that $K$ is divisible by $4$.

The following lower bound for $A+A-A$ is essentially copied from an argument of \cite{HRNR}.

\begin{Claim} \label{claim1}
\[
|A+A-A| \gg \frac{|A||D'|}{\log|A|}.
\]
\end{Claim}

\begin{proof}[Proof of Claim \ref{claim1}]
Decompose $D'= D_{big}' \sqcup D_{small}'$ into two disjoint parts with sizes as equal as possible (in particular $|D_{small}| - 1 \leq |D_{big}| \leq |D_{small}| $) such that every element of $D_{big}'$ is larger than every element of $D_{small}'$. Let $A' \subset A$ denote the set
\[
A'= \{a_i \in A : a_{i+1} - a_i \in D_{big}' \} = \bigcup_{d \in D_{big}'} A_d.
\]
Observe that $|A'| \geq |D_{big}'|K \gg |D'|K $. Also, for each $a_i \in A'$, note that
\[
a_i + D_{small}' \subset (a_i,a_{i+1}).
\]
This implies the existence of $|D'|/2$ elements of $A+A-A$ in the interval $(a_i,a_{i+1})$. Summing over all $a \in A'$ and applying \eqref{D'Kbound}, it follows that
\[
|A+A-A| \gg |D'|^2K \gg \frac{|A||D'|}{\log|A|}.
\]
\end{proof}

Claim \ref{claim1} allows us to henceforth assume that 
\begin{equation} \label{Kmassume}
K \geq 8m.
\end{equation}
Indeed, suppose that this is not the case. Then combining Claim \ref{claim1} with \eqref{D'Kbound}, it follows that
\[
|A+A-A| \gg \frac{|A|^2}{m( \log|A|)^2},
\]
which is strong enough to imply the intended bound \eqref{product}.

Next, we make some additional refinements to the sets $D'$ and $A_d$. This is done in order to set up a lower bound for iterated sums of $f(A)$.

Recall from the definition at the beginning of this section that $f_d(x):=f(x+d)-f(x)$. For a given $d \in D'$, write
\[
A_d= \{ a_{i_1}, \dots, a_{i_K} \}
\]
where $i_1 < \dots < i_K$.  Let $A_d'$ denote the smallest half of $A_d$, so
\[
A_d'= \{ a_{i_1}, \dots, a_{i_{K/2}} \}.
\]
Observe that, since $f$ is strictly convex,
\[
f_d(a_{i_1}) < \dots < f_d(a_{i_K}).
\]
Define
\[
t_d:=f_d(a_{i_{K/2}}).
\]
That is, $t_d$ is the largest element of $f_d(A_d')$. Hence $f_d(A_d') \subset (0,t_d]$.

In the forthcoming argument, it will be helpful to force the elements of $f_d(A_d')$ to be slightly closer to each other. For each  $a_{i_j} \in A_d'$, the image $f_d(a_{i_j})$ belongs to one of the intervals $(0,t_d/2]$ or $(t_d/2,t_d]$, and so it follows that at least half of the elements $f_d(a_{i_j})$ are in the same interval. In order to simplify the notation further, let us assume that the first of these intervals is well populated. In particular, if we define
\[
A_d'':= \{ a_{i_1}, \dots, a_{i_{K/4}} \},
\]
then $f_d(A_d'') \subset (0,t_d/2]$, and it follows that
\[
2f_d(A_d'')-2f_d(A_d'') \subset (t_d,t_d).
\]
We dyadically decompose $D'$ according to the size of the difference set $2f_d(A_d'')-2f_d(A_d'')$. Observe that
\[
|D'|= \sum_{j =1}^{ 4\log |A|}\sum_{\stackrel{d \in D' :}{ 2^{j-1}K \leq |2f_d(A_d'')-2f_d(A_d'')| <2^{j}K}} 1.
\]
Therefore, there exists $D'' \subset D'$ such that 
\[
|D''| \gg \frac{|D'|}{\log |A|}
\]
and, for all $d \in D''$,
\begin{equation} \label{doubling}
LK \leq |2f_d(A_d'')-2f_d(A_d'')| < 2LK,
\end{equation}
where $L \geq 1$ is some fixed integer. For the remainder of the proof, we work only with the differences $d \in D''$.

The following claim builds on work of \cite{HRNR}, where a basic version of this argument was used to prove the bound $|2f(A)-f(A)| \gg \frac{ |A|K}{\log |A|}$, which can then be combined with the bound of Claim \ref{claim1} to prove the bound \eqref{claim2}. The difference in approach for the following claim is that we squeeze an iterated difference set into the gaps, which results in some improvement unless the previously defined parameter $L$ is constant. A similar approach was recently used in \cite{RNZ}.

\begin{Claim} \label{claim2}
    \[
    |5f(A)-4f(A)| \gg \frac{ |A|KL}{(\log |A|)^2}.
\]
\end{Claim}

\begin{proof}[Proof of Claim \ref{claim2}]

Let $d \in D''$ and let $E_d$ denote the non-negative elements of $2f_d(A_d'')- 2f_d(A_d'')$. So, $E_d \subset [0,t_d)$, and since the difference set is symmetric,
\begin{equation} \label{Edbound}
|E_d| \gg |2f_d(A_d'')-2f_d(A_d'')| \geq LK.
\end{equation}
Now, let $K/2 <j \leq K$ and consider the set of sums
\begin{equation} \label{inclusion}
f(a_{i_j}) + E_d \subset [f(a_{i_j}), f(a_{i_j})+t_d) \subset [f(a_{i_j}), f(a_{i_j+1})).
\end{equation}
The first inclusion above is valid because $E_d \subset [0,t_d)$. The second is equivalent to the inequality
\begin{equation} \label{equiv}
f(a_{i_j})+t_d \leq f(a_{i_j+1}).
\end{equation}
Since $a_{i_j} \in A_d$, it follows that $a_{i_j+1}=a_{i_j}+d$, and so the inequality \eqref{equiv} can be written as $f_d(a_{i_j})\geq t_d =f_d(a_{i_{K/2}})$. This is valid, since $f_d$ is monotone increasing and $j >K/2$. We have thus verified the inclusion \eqref{inclusion}.

On the other hand, it follows from the construction of the set $E_d$ that $E_d \subset 4f(A)-4f(A)$, and therefore
\[
f(a_{i_j}) + E_d \subset  (5f(A)-4f(A) ) \cap  [f(a_{i_j}), f(a_{i_j+1})).
\]
We have thus identified $|E_d|$ elements of $5f(A)-4f(A)$ lying in between consecutive elements of $f(A)$. Recall from \eqref{Edbound} that $|E_d| \gg LK$. We repeat this process for each $K/2<j<K$, and then for each $d \in D''$. It follows from \eqref{D'Kbound} that
\[
|5f(A)-4f(A)| \gg LK^2|D''| \gg \frac{LK^2|D'|}{\log|A|} \gg \frac{LK|A|}{(\log|A|)^2} ,
\]
which completes the proof of the claim.

\end{proof}

Combining the previous two claims, we are done unless $L$ is very small. However, if this is the case, the next claim allows us to win on the $A$ side. The idea here is that, if $L$ really is small, then this implies that some additive structure is buried in $f(A)$, which in turn should imply that $A$ grows under addition.

\begin{Claim} \label{claim3}
\[
|8A-7A| \gg  \frac{ K^4}{(Lm)^{11}(\log |A|)^{25}}.
\]
\end{Claim}

\begin{proof}[Proof of Claim \ref{claim3}]

It follows from \eqref{doubling} that
\[
|f_d(A_d'')+ f_d(A_d'') - f_d(A_d'') | \leq 2LK=8L|f_d(A_d'')|.
\]

Now, recall the assumption from the statement of the theorem that the function $f_d^{-1}:J \rightarrow I$ exists and that its first three derivatives have a total of at most $m$ zeroes. We use these zeroes to divide the codomain of $f_d^{-1}$ into $m-1$ pairwise disjoint subintervals $J_1, \dots, J_{m+1} \subset J$ , such that on each of the subintervals, all of the first three derivates are non-zero.

Let
\[
I_1=f_d^{-1}(J_1), \dots , I_{m+1}=f_d^{-1}(J_{m+1}).
\]
Observe that the intervals $I_1, \dots, I_{m+1} \subset I$ are pairwise disjoint, and that the union
$
\bigcup_{i=1}^{m+1} I_i
$
is equal to the the set $I$ with at most $m$ points removed. Since $A_d''  \subset I$, it follows that there is some index $i$ such that 
\[
|A_d'' \cap I_i| \geq \frac{|A_d''| - m}{m+1} \geq \frac{ K}{16m}.
\]
The last inequality above makes use of \ref{Kmassume}. Let
\[
A_d''':= A_d'' \cap I_i
\]
and define $g: J_i \rightarrow I$ to be the function $f_d^{-1}$ with the domain restricted to $J_i$. Since the first three derivatives of $g$ are all non-zero in $J_i$, the function $g$ is by definition $3$-convex.

It follows from \eqref{doubling} that
\begin{align*}
|g(A_d''') + g(A_d''') - g(A_d''') & \leq |f_d(A_d'')+ f_d(A_d'') - f_d(A_d'') |
\\& \leq 2LK  \leq 32Lm|g(A_d''')|.
\end{align*}

Apply Theorem \ref{thm:BOM} with
\[
f=g , \,\,\,\,\,\, k=3, \,\,\,\,\,\, A=f_d(A_d'''), \,\,\,\,\,\, K=32Lm.
\]
It follows that
\begin{align*}
    |8A-7A| & \geq |8A_d'''-7A_d'''|
    \\& \geq \frac{ |A_d''|^4}{(C'mL)^{11}(\log |A|)^{25}}
    \\& \gg  \frac{ K^4}{(Lm)^{11}(\log |A|)^{25}}.
\end{align*}
\end{proof}

The final task is to combine the inequalities from the previous three claims to complete the proof. We have
\begin{align*}
|8A-7A|^{16}|5f(A)-4f(A)|^{11} & \geq 
|8A-7A||A+A-A|^{15}|5f(A)-4f(A)|^{11}  \\ &\gg \frac{ K^4}{(Lm)^{11}(\log |A|)^{25}} \cdot \left ( \frac{|A||D'|}{\log |A|}\right )^{15} \cdot \left ( \frac{ |A|KL}{(\log |A|)^2}\right)^{11} 
\\ & = \frac{|A|^{26}K^{15}|D'|^{15}}{m^{11}(\log|A|)^{62}}
\\& \gg \frac{|A|^{41}}{ m^{11}(\log|A|)^{77}}
\\& \gg_m \frac{|A|^{41}}{(\log|A|)^{77}}
\end{align*}
The last inequality above is an application of \eqref{D'Kbound}. This completes the proof of \eqref{product}, and therefore also that of Theorem \ref{thm:maingen}.

\end{proof}

\section{Applying Theorem \ref{thm:maingen} for some particular convex functions of interest}

\begin{Corollary} \label{cor:log}
     Suppose that $A \subset (0, \infty)$ is a finite set. Then
    \[
    \max \left \{|8A-7A|, \left|\frac{AAAAA}{AAAA}\right| \right \} \gtrsim |A|^{\frac{3}{2} +\frac{1}{54}}.
    \]
    
\end{Corollary}
\begin{proof}
    Apply Theorem \ref{thm:maingen} with $f(x)= \ln x$. The function $f_d:(0, \infty) \rightarrow (0,\infty)$ is defined by
    \[
    f_d(x)= \ln \left ( \frac{x+d}{x} \right).
    \]
    Its inverse is 
    \[
    f_d^{-1}:(0, \infty) \rightarrow (0,\infty), \,\,\,\,\,\, f_d^{-1}(x)=\frac{d}{e^x-1}.
    \]
    A direct calculation shows that the first three derivatives of $f_d^{-1}$ are non-zero in $(0, \infty)$. Indeed, the first three derivatives are
    \begin{align*}
       (f_d^{-1})^{(1)}(x) & = \frac{-de^x}{(e^x-1)^2},
       \\ (f_d^{-1})^{(2)}(x) &=  \frac{de^x (e^x+1)}{(e^x-1)^3} ,
       \\ (f_d^{-1})^{(3)}(x) &=  \frac{-de^x (e^{2x}+4e^x+1)}{(e^x-1)^4} .
    \end{align*}
\end{proof}

In this case, the condition that $A$ consists only of positive elements is not significant. For an arbitrary finite set $A \subset \mathbb R$, at least $\frac{|A|-1}{2}$ of the elements have (strictly) the same sign. let $A' \subset A$ be a set with size at least $\frac{|A|-1}{2}$ such that all elements of $A'$ have the same sign. If $A' \subset (0, \infty)$ then apply Corollary \ref{cor:log} to $A'$. Otherwise, it can be applied to $-A'$ to give the same result.

\begin{Corollary} \label{cor:shifts}
     Let $\lambda$ be any strictly positive real number and suppose that $X \subset (0, \infty)$ is a finite set. Then
    \[
    \max \left \{\left|\frac{X^{(8)}}{X^{(7)}} \right |, \left|\frac{(X+\lambda)^{(5)}}{(X+\lambda)^{(4)}}\right| \right \} \gtrsim |X|^{\frac{3}{2} +\frac{1}{54}}.
    \]
    
\end{Corollary}

\begin{proof}
    Apply Theorem \ref{thm:maingen} with $A= \ln X$ and $f(x)= \ln (e^x+\lambda)$. The function $f_d: \mathbb R \rightarrow (0,d)$ is defined by
    \[
  f_d(x)=  \ln \left ( \frac{e^{x+d}+\lambda}{e^x+ \lambda} \right).
    \]
    Its inverse is
    \[
    f_d^{-1}:(0, d) \rightarrow \mathbb R, \,\,\,\,\,\, f_d^{-1}(x)= \ln \left ( \frac{\lambda(e^{x}-1)}{e^d-e^x} \right).
    \]
   The first three derivatives of $f_d^{-1}$ are
    \begin{align*}
       (f_d^{-1})^{(1)}(x) & = \frac{(e^d-1)e^x}{(e^x-1)(e^d-e^x)},
       \\ (f_d^{-1})^{(2)}(x) &=  \frac{(e^d-1)e^x(e^{2x}-e^d)}{(e^x-1)^2(e^d-e^x)^2},
       \\ (f_d^{-1})^{(3)}(x) &=  \frac{(e^d-1)e^x[e^{4x}+(e^d+1)e^{3x}-6e^de^{2x}+(e^{2d}+e^d)e^x+e^{2d}]}{(e^x-1)^3(e^d-e^x)^3}.
    \end{align*}
    The three derivatives have a combined total of at most $5$ zeroes in $(0,d)$, and so Theorem \ref{thm:maingen} can indeed be applied with $m=5$.
\end{proof}

Unfortunately, the condition that $A$ consists of a set of positive reals appears to be a more meaningful restriction for Corollary \ref{cor:shifts}, and this cannot be easily removed by applying the result to a dilate of the set.

We have also directly verified that the additional condition of Theorem \ref{thm:maingen} is valid for other notable functions such as $f(x)=e^x$ and $f(x)=x^3$. The details of these calculations are omitted. It appears likely that Theorem \ref{thm:maingen} also holds for $f(x)=x^d$ with $d \geq 3$ an integer. However, the additional condition does not hold for the quadratic function $f(x)=x^2$.

We now use Corollary \ref{cor:shifts} to prove Theorem \ref{thm:main2}. We state and prove a more quantitatively precise version of the theorem below.

\begin{Corollary}
    Let $A$ be a set of positive real numbers and suppose that $|AA| \leq K|A|$. Then, for all $t \in \mathbb R$ such that $t \neq 0$,
    \[
    |\{ (a,b) \in A \times A : a-b=t \}| \lesssim K^{\frac{405}{41}} |A|^{\frac{2}{3}- \frac{1}{123}} .
    \]
\end{Corollary}

\begin{proof}
Write 
\[
r_{A-A}(t):= |\{ (a,b) \in A \times A : a-b=t \}|.
\]
 Since $r_{A-A}(t)=r_{A-A}(-t)$, we may assume without loss of generality that $t$ is positive.

Denote $A(t)=A \cap (A-t)$ and observe that $|A(t)|=r_{A-A}(t)$. Note also that $A(t), A(t)+t \subset A$. Apply Corollary \ref{cor:shifts} with 
\[
X=A(t), \,\,\,\, \lambda=t.
\]
It follows that
\[
\left|\frac{A^{(8)}}{A^{(7)}} \right | \geq \max \left \{\left|\frac{A(t)^{(8)}}{A(t)^{(7)}} \right |, \left|\frac{(A(t)+t)^{(5)}}{(A(t)+t)^{(4)}}\right| \right \} \gtrsim |A(t)|^{\frac{3}{2}+ \frac{1}{54}}.
\]
However, using Theorem \ref{thm:Plun} in the multiplicative setting and the assumption that $|AA| \leq K|A|$, it follows that
\[
|A(t)|^{\frac{3}{2}+ \frac{1}{54}} \lesssim K^{15}|A|.
\]
A rearrangement of this inequality completes the proof.

\end{proof}

%\orn{A short section here about how we can use Plunnecke to rreduce the number of variables slightly, and also to give a bound involving only sums and products.}

\section{Proof of Theorem \ref{thm:main1}}

In this section, we will prove Theorem \ref{thm:main1}. In fact, we will prove the following more general statement, from which Theorem \ref{thm:main1} follows by setting $f(x) = \ln x$. A concrete value for the constant $c$ from Theorem \ref{thm:main1} is given, namely $c= \frac{1}{162}$, although the task of optimising this constant is not pursued to its fullest here.

\begin{Theorem} \label{thm:main1gen}
 Let $a \in \mathbb R$ and let $I$ denote the interval $I=(a, \infty)$. Suppose that $A \subset I$ is a finite set. Suppose that $f:(a, \infty) \rightarrow \mathbb R$ is a strictly convex or concave function.  Suppose also that, for any $d \in (0, \infty)$, the function $f_d^{-1}: J \rightarrow I$ is $3$-convex. Then
    \[
    \max \{|16A|, |13f(A)| \} \gtrsim |A|^{\frac{3}{2} +\frac{1}{162}}.
    \]
\end{Theorem}

\begin{proof}
An application of Theorem \ref{thm:ENR} with $k=8$ gives the bound
\begin{equation} \label{ENRprod}
|8A| |8f(A)| \gg |A|^{3-\frac{1}{128}}.
\end{equation}
It follows that we may assume that 
\begin{equation} \label{assume}
|8A| \geq |A|^{\frac{3}{2} - \frac{1}{64}}.
\end{equation}
Indeed, if \eqref{assume} does not hold then \eqref{ENRprod} gives the bound $|8f(A)| \gg |A|^{\frac{3}{2} + \frac{1}{128}} $, which is stronger than the result claimed in the statement of Theorem \ref{thm:main1gen}.

Two applications of Theorem \ref{thm:Ruzsa} give the bounds
\[
|8A-7A| \leq |8A-8A| \leq \frac{|8A+8A|^2}{|8A|} = \frac{|16A|^2}{|8A|}
\]
and similarly
\[
|5f(A)-4f(A)| \leq  \frac{|13f(A)|^2}{|8f(A)|}.
\]
Combining these bounds with \eqref{product}, it follows that
\[
\frac{|16A|^{32}}{|8A|^{16}}  \cdot \frac{|13f(A)|^{22}}{|8f(A)|^{11}} \geq  |8A-7A|^{16}|5f(A)-4f(A)|^{11}  \gtrsim |A|^{41}.
\]
Rearranging this inequality and then applying \eqref{ENRprod} yields
\[
|16A|^{32}|13f(A)|^{22} \gtrsim |A|^{41 } (|A|^{3-\frac{1}{128}})^{11} |8A|^5.
\]
Applying \eqref{assume} and then tidying things up gives
\[
|16A|^{32}|13f(A)|^{22} \gtrsim  |A|^{81 + \frac{1}{2}- \frac{21}{128}} \geq |A|^{81 + \frac{1}{3}}.
\]
Finally, the claimed bound follows from the pigeonhole principle.
\end{proof}

\section*{Acknowledgements}

The author was supported by the Austrian Science Fund FWF Project P 34180. Thanks to Brandon Hanson, Audie Warren and Dmitrii Zhelezov for helpful conversations.

\end{document}